\documentclass[12pt]{amsart}

\usepackage{algorithmic}
\usepackage[section]{algorithm}

\usepackage{fullpage}

\usepackage{amsfonts,amssymb,amscd,amsmath,enumerate,verbatim,color,xcolor}
\usepackage[latin1]{inputenc}
\usepackage{amscd}
\usepackage{latexsym}
\usepackage{tikz}
\usepackage{subcaption}
\usepackage{graphicx} 

\usepackage{mathptmx}

\usepackage{hyperref}
\hypersetup{colorlinks,linkcolor=blue ,citecolor=blue, urlcolor=blue}
\def\NZQ{\mathbb}               

\def\QQ{{\NZQ Q}}
\def\ZZ{{\NZQ Z}}
\def\RR{{\NZQ R}}

\def\KK{{\NZQ K}}

\def\frk{\mathfrak}               

\def\Phi{{\frk N}}
\def\ab{{\mathbf a}}

\def\eb{{\mathbf e}}
\def\tb{{\mathbf t}}

\def\xb{{\mathbf x}}

\def\opn#1#2{\def#1{\operatorname{#2}}} 
\opn\gr{gr}

\def\Gc{{\mathcal G}}
\def\Fc{{\mathcal F}}

\newtheorem{Theorem}{Theorem}[section]
\newtheorem{Lemma}[Theorem]{Lemma}
\newtheorem{Corollary}[Theorem]{Corollary}
\newtheorem{Proposition}[Theorem]{Proposition}

\theoremstyle{definition}
\newtheorem{Remark}[Theorem]{Remark}

\newtheorem{Example}[Theorem]{Example}

\newtheorem{Definition}[Theorem]{Definition}

\let\epsilon\varepsilon
\let\phi=\varphi
\let\kappa=\varkappa
\opn\dis{dis}
\opn\height{height}
\opn\dist{dist}
\def\pnt{{\raise0.5mm\hbox{\large\bf.}}}

\opn\Lex{Lex}
\opn\conv{conv}

\numberwithin{equation}{section}

\title{Examining Kempe equivalence via commutative algebra}
\author{Hidefumi Ohsugi and Akiyoshi Tsuchiya}
\address{Hidefumi Ohsugi,
	Department of Mathematical Sciences,
	School of Science,
	Kwansei Gakuin University,
	Sanda, Hyogo 669-1330, Japan} 
\email{ohsugi@kwansei.ac.jp}

\address{Akiyoshi Tsuchiya,
Department of Information Science,
Faculty of Science,
Toho University,
2-2-1 Miyama, Funabashi, Chiba 274-8510, Japan} 
\email{akiyoshi@is.sci.toho-u.ac.jp}


\begin{document}

\maketitle
\begin{abstract}
Kempe equivalence is a classical and important notion on vertex coloring in graph theory.
In the present paper, we introduce several ideals associated with graphs and provide a method for determining whether two $k$-colorings are Kempe equivalent via commutative algebra.
Moreover, we give a way to compute all $k$-colorings of a graph up to Kempe equivalence by virtue of the algebraic technique on Gr\"{o}bner bases.
As a consequence, the number of $k$-Kempe classes can be computed by using Hilbert functions. 
Finally,  
we introduce several algebraic algorithms related to Kempe equivalence.
\end{abstract}

\section{Introduction}

A {\em $k$-coloring} $f$ of a graph $G$ on the vertex set $[d]:=\{1,2,\ldots,d\}$ is a map from $[d]$ to $[k]$ such that
$f(i) \neq f(j)$ for all $\{i,j\} \in E(G)$.
The smallest integer $\chi(G)$ such that $G$ has a $\chi(G)$-coloring is called the \textit{chromatic number} of $G$.
Given a $k$-coloring $f$ of $G$, and integers $1 \le i < j \le k$,
let $H$ be a connected component of the induced subgraph of $G$ 
on the vertex set $f^{-1}(i) \cup f^{-1}(j)$.
Then we can obtain a new $k$-coloring $g$ of $G$ by setting
$$
g(x) =
\begin{cases}
    f(x) & x \notin H,\\
    i & x \in H \mbox{ and } f(x) =j,\\
    j & x \in H \mbox{ and } f(x) =i.\\
\end{cases}
$$
We say that $g$ is obtained from $f$ by a {\em Kempe switching}.
Two $k$-colorings $f$ and $g$ of $G$ are called {\em Kempe equivalent}, denoted by $f \sim_k g$, if there exists
a sequence $f_0,f_1,\ldots,f_s$ of $k$-colorings of $G$ such that $f_0=f$, $f_s=g$,
and $f_i$ is obtained from $f_{i-1}$ by a Kempe switching.
Let denote $\mathcal{C}_k(G)$ the set of all $k$-colorings of $G$. Then $\sim_k$ is an equivalence relation on $\mathcal{C}_k(G)$.
The equivalence classes of $\mathcal{C}_k(G)$ by $\sim_k$ are called the $k$-Kempe classes.
We denote ${\rm kc}(G,k)$ the quotient set $\mathcal{C}_k(G) / \sim_{k}$ and denote ${\rm Kc}(G,k)$ the number of $k$-Kempe classes of $G$, namely ${\rm Kc}(G,k)=|{\rm kc}(G,k)|$.
Kempe switchings were introduced by Kempe in the false proof of the $4$-Color Theorem. However, his idea is powerful in graph coloring theory. Recently, many researchers have studied Kempe switchings and Kempe equivalence.
See \cite{Moh} for an overview of Kempe equivalence.

Given a graph $G$, let $S(G)$ be the set of all stable sets of $G$.
The \textit{stable set ideal} $I_G$ of $G$ is a toric ideal arising from $S(G)$ of a polynomial ring $R[G]:=\KK[x_S \mid S \in S(G)]$ over a field $\KK$.
In \cite{OTkempe}, the authors 
showed that $I_G$ is generated by binomials $\xb_f-\xb_g$ associated with $k$-colorings $f$ and $g$ of a replication graph
of an induced subgraph of $G$, and
found a relationship between Kempe equivalence on $G$ and an algebraic property of $I_G$.
In particular, by using the proof of \cite[Theorem 1.3]{OTkempe}, we can examine if two $k$-colorings of $G$ are Kempe equivalent by using $I_G$. However, $I_G$ has too much information for this purpose.
In the present paper, we introduce a
simpler ideal $J_G$, which is generated by binomials $\xb_f-\xb_g$ associated with $2$-colorings $f$ and $g$ of an induced subgraph of $G$, to determine Kempe equivalence on $G$. We call $J_G$ the \textit{$2$-coloring ideal} of $G$.
Then our first main result is the following:
\begin{Theorem}
\label{thm:K_equiv}
    Let $G$ be a graph on $[d]$ and let $f,g$ be $k$-colorings of $G$. Then $f \sim_k g$ if and only if 
    $\xb_f-\xb_g \in J_G$.
\end{Theorem}

Next, we compute all $k$-colorings of a graph $G$ up to Kempe equivalence by virtue of the algebraic technique on Gr\"{o}bner bases.
Namely, a complete representative system for ${\rm kc}(G,k)$ is given.
For this,
we introduce another ideal $K_G$ defined by
\[
K_G:= J_G + M_G,
\]
where
\[
M_G := \langle x_S x_T \mid S, T \in S(G), S \cap T \neq \emptyset\rangle.
\]
The ideal $K_G$ is called the \textit{Kempe ideal} of $G$.
Then our second main result is the following:
\begin{Theorem}
\label{thm:K-class}
        Let $G$ be a graph on $[d]$ and $<$ a monomial order on $R[G]$, and let $\{\xb_{f_1},\dots,\xb_{f_s}\}$ be the set of all standard monomials of degree $k$ with respect to the initial ideal ${\rm in}_<(K_G)$.
    Then 
\[
\{f_1, \dots, f_s\} \cap \mathcal{C}_k(G)
\]
is a complete representative system
    for ${\rm kc}(G,k)$.
\end{Theorem}
As a consequence, the number of $k$-Kempe classes ${\rm Kc}(G,k)$ can be computed by Hilbert functions (Corollary \ref{cor:K_class}).

Finally,  
by using Theorems \ref{thm:K_equiv} and \ref{thm:K-class} and techniques on Gr\"{o}bner bases, we introduce several algorithms related to Kempe equivalence. Specifically, our algorithms perform:
\begin{enumerate}
    \item Determination of Kempe equivalence (Algorithm \ref{one});
    \item Computation of a complete representative system for ${\rm kc}(G,k)$ (Algorithm \ref{kanzen});
    \item Enumeration of $k$-colorings that are Kempe equivalent (Algorithm \ref{enumerate});
    \item Construction of a sequence of Kempe switchings between two Kempe equivalent $k$-colorings (Algorithm \ref{alg1}).
\end{enumerate}

The present paper is organized as follows: In Section \ref{sect:stable}, we will recall the definition of stable set ideals and explain a relationship between stable set ideals and Kempe equivalence.
Section \ref{sect:GB} will give a brief introduction to Gr\"{o}bner bases.
In Section \ref{sect:2-color}, we will define $2$-coloring ideals and see their algebraic properties.
In Section \ref{sec:kempe}, a proof of Theorem \ref{thm:K_equiv} will be given.
In Section \ref{sect:k-class}, we will define Kempe ideals and prove Theorem \ref{thm:K-class}.
Finally, in Section \ref{sect:algo}, we will introduce several algorithms related to Kempe equivalence.
\section{Stable set ideals}
\label{sect:stable}
In this section, we define stable set rings and explain a relationship between stable set ideals and Kempe equivalence.
Let $G$ be a graph on the vertex set $[d]$ with the edge set $E(G)$.
Given a subset $S \subset [d]$,
let $G[S]$ denote the induced subgraph of $G$ on the vertex set $S$.
A subset $S \subset [d]$ is called a {\em stable set} (or an {\em independent set}) of $G$
if $\{i,j\} \notin E(G)$ for all $i,j \in S$ with $i \neq j$.
Namely, a subset $S \subset [d]$ is stable if and only if $G[S]$ is an empty graph.
In particular, the empty set $\emptyset$ and any singleton $\{i\}$ with $i \in [d]$
are stable.
Denote $S(G)=\{S_1,\ldots,S_n\}$ the set of all stable sets of $G$.
Given a subset $S \subset [d]$, we associate the $(0,1)$-vector $\rho(S)=\sum_{j \in S} \eb_j$. Here $\eb_j$ is the $j$th unit coordinate vector in $\RR^d$. For example, $\rho(\emptyset)=(0,\ldots,0) \in \RR^d$.
Let $\KK[\tb,s]:=\KK[t_1,\ldots,t_d,s]$ be the polynomial ring in $d+1$ variables  over a field $\KK$.
Given a nonnegative integer vector 
$\ab=(a_1,\ldots,a_d) \in  \ZZ_{\geq 0}^d$, we write 
$\tb^{\ab}:=t_1^{a_1} t_2^{a_2}\cdots t_d^{a_d} \in \KK[\tb,s]$.
The \textit{stable set ring} of $G$ is 
\[
\KK[G]:=\KK[\tb^{\rho(S_1)} s,\ldots, \tb^{\rho(S_n)} s] \subset \KK[\tb,s]. 
\]
We regard $\KK[G]$ as a homogeneous algebra by setting each $\deg (\tb^{\rho(S_i)} s)=1$.
Note that $\KK[G]$ is a toric ring.
Let $R[G]=\KK[x_{S_1},\ldots,x_{S_n}]$ denote the polynomial ring in $n$ variables over  $\KK$ with each $\deg(x_{S_i})=1$.
The \textit{stable set ideal} of $G$ is the kernel of the surjective homomorphism 
$\pi:R[G] \to \KK[G]$ defined by $\pi(x_{S_i})=\tb^{\rho(S_i)}s$ for $1 \leq i \leq n$.
Note that $I_G$ is a toric ideal, and hence a prime ideal generated by homogeneous binomials.
The toric ring $\KK[G]$ is called \textit{quadratic}
if $I_G$ is generated by quadratic binomials.
We say that ``$I_G$ is generated by quadratic binomials'' even if $I_G =\{0\}$ (or equivalently, $G$ is complete).
It is easy to see that a homogeneous binomial
$x_{S_{i_1}} \cdots x_{S_{i_r}} - x_{S_{j_1}} \cdots x_{S_{j_r}} \in R[G]$ belongs to $I_G$
if and only if 
$\bigcup_{\ell=1}^r S_{i_\ell}=\bigcup_{\ell=1}^r S_{j_\ell}$ 
as multisets.
See, e.g., \cite{HHO} for details on toric rings and toric ideals.

We can describe a system of generators of $I_G$ in terms of $k$-colorings.
Given a graph $G$ on the vertex set $[d]$, and 
$\ab = (a_1,\ldots,a_d) \in \ZZ_{\ge 0}^d$,
let $G_\ab$ be the graph obtained from $G$ by replacing each vertex $i \in [d]$ 
with a complete graph $G^{(i)}$ of $a_i$ vertices (if $a_i =0$, then just delete the vertex $i$),
and joining all vertices $x \in G^{(i)}$ and $y \in G^{(j)}$ such that $\{i,j\}$ is an edge of $G$.
In particular, if $\ab =(1,\ldots,1)$, then $G_\ab = G$.
If $\ab = {\bf 0}$, then $G_\ab$ is the null graph (a graph without vertices). 
In addition, if $\ab$ is a $(0,1)$-vector, namely, $\ab \in \{0,1\}^d$, then $G_\ab$ is an induced subgraph of $G$.
If $\ab$ is a positive vector, then $G_\ab$ is called a 
\textit{replication graph} of $G$.
In general, $G_\ab$ is a replication graph of an induced subgraph of $G$.
Given a $k$-coloring $f$ of $G_{\ab}$ with $\ab \in \ZZ^d_{\geq 0}$,
we associate $f$ with a monomial
\[
\xb_f :=x_{S_{i_1}} \cdots x_{S_{i_k}} \in R[G], 
\]
where
$S_{i_\ell} =  \{j \in [d] \mid G^{(j)} \cap f^{-1}(\ell) \ne \emptyset\}$
for $\ell=1,2,\ldots, k$.
Conversely, let $m=x_{S_{i_1}} \cdots x_{S_{i_k}} \in R[G]$ be a monomial of degree $k$.
Then, for $\ab=(a_1,\ldots,a_n)$ with $a_p = |\{ \ell : p \in S_{i_\ell}\}|$,
there exists a $k$-coloring $f$ of $G_{\ab}$ such that $\xb_f = m$ (see \cite[Lemma 3.2]{OTkempe}).
For example, we consider the graphs $G$ and $G_{\ab}$ with $\ab=(2,1,0,2)$  as follows:
\begin{figure}[H]
\centering

\begin{subfigure}{0.45\textwidth}
\centering
\begin{tikzpicture}
\node[draw, shape=circle] (1) at (-1,1) {$1$};	
\node[draw, shape=circle] (2) at (0,2) {$2$}; 
\node[draw, shape=circle] (4) at (1,1) {$4$}; 
\node[draw, shape=circle] (3) at (0,0) {$3$}; 
\draw (1)--(2);
\draw (1)--(3);
\draw (2)--(3);
\draw (2)--(4);
\end{tikzpicture}
\caption{$G$}
\end{subfigure}
\quad
\begin{subfigure}{0.45\textwidth}
\centering
\begin{tikzpicture}

\node[draw, shape=circle] (11) at (-1,1.5) {\tiny{$1_1$}};
\node[draw, shape=circle] (12) at (-1,0.5) {\tiny{$1_2$}};

\node[draw, shape=circle] (21) at (0.5,2.3) {\tiny{$2_1$}};


\node[draw, shape=circle] (41) at (2.0,1.5) {\tiny{$4_1$}};
\node[draw, shape=circle] (42) at (2.0,0.5) {\tiny{$4_2$}};

\draw (11)--(21);
\draw (12)--(21);

\draw (21)--(41);
\draw (21)--(42);

\draw (11)--(12);

\draw (41)--(42);

\end{tikzpicture}
\caption{$G_{\ab}$ with $\ab=(2,1,0,2)$}
\end{subfigure}

\end{figure}
We define a $4$-coloring $f$ of $G_{\ab}$  by
\[
f(i)=\begin{cases}
    1 & i \in \{1_1,4_1\},\\
    2 & i \in \{2_1\},\\
    3 & i \in \{1_2\},\\
    4 & i \in \{4_2\}.
\end{cases}
\]
Then since $S_{i_1}=\{1,4\},S_{i_2}=\{2\}, S_{i_3}=\{1\},S_{i_4}=\{4\}$, one has \[\xb_f=x_{\{1,4\}} x_{\{2\}} x_{\{1\}} x_{\{4\}}.\]
On the other hand, we can obtain the $4$-coloring $f$ (up to exchanging colors and
exchanging the coloring of vertices in each clique $G^{(i)}$ of $G_\ab$) from the monomial
$x_{\{1,4\}} x_{\{2\}} x_{\{1\}} x_{\{4\}}$ as follows: for each variable $x_{S}$, we assign one color to a single copy $i_j$ of each vertex $i \in S$.

Note that, for $k$-colorings $f$ and $g$ of an induced subgraph of $G$, 
$\xb_f = \xb_g$ if and only if $g$ is obtained from $f$ by permuting colors.
It is easy to see that $f \sim_k g$ if $g$ is obtained from $f$ by permuting colors.
In this paper, we identify $f$ and $g$ if $g$ is obtained from $f$ by permuting colors.
%
%
Then we can describe a system of generators of $I_G$ as follows:

\begin{Proposition}[{\cite[Theorem 3.3]{OTkempe}}]
\label{prop:gen}
Let $G$ be a graph on $[d]$. Then one has 
\[
\xb_f- \xb_g \in I_G \iff  \mbox{$f$ and $g$ are $k$-colorings of $G_\ab$ with $\ab \in \ZZ^d_{\geq 0}$ and $k \geq \chi(G_{\ab})$} 
\]
and
\begin{align*}
I_G &= \langle \xb_f- \xb_g \mid \mbox{$f$ and $g$ are $k$-colorings of $G_\ab$ with $\ab \in \ZZ^d_{\geq 0}$ and $k \geq \chi(G_{\ab})$}  \rangle\\
&= \langle \xb_f- \xb_g \mid \mbox{$f$ and $g$ are k-colorings of $G_\ab$ with $\ab \in \{0,1,\dots,k\}^d$ and $k \geq \chi(G_{\ab})$}  \rangle.
\end{align*}
\end{Proposition}

Now, we explain a relationship between $I_G$ and Kempe equivalence.

\begin{Proposition}[{\cite[Theorem 1.3]{OTkempe}}]
\label{prop:quad}
Let $G$ be a graph on $[d]$.
Then $I_G$ is generated by quadratic binomials if and only if for any $\ab \in \ZZ^d_{\geq 0}$ and $k \geq \chi(G_{\ab})$, all $k$-colorings of $G_{\ab}$ are Kempe equivalent, namely, one has
\[
{\rm Kc}(G_{\ab}, k)=1.
\]
\end{Proposition}

We consider the ideal $\langle [I_G]_2 \rangle$ of $R[G]$ generated by all quadratic binomials of $I_G$, namely, 
\begin{align*}
\langle [I_G]_2 \rangle &=\langle \xb_f- \xb_g \mid \mbox{$f$ and $g$ are 2-colorings of $G_\ab$ with $\ab \in \ZZ^d_{\geq 0}$}  \rangle \\
    &=\langle \xb_f- \xb_g \mid \mbox{$f$ and $g$ are 2-colorings of $G_\ab$ with $\ab \in \{0,1,2\}^d$}  \rangle \subset R[G].
\end{align*}

\begin{Proposition}
\label{prop:exam_kempe}
    Let $G$ be a graph on $[d]$ and $f,g$ $k$-colorings of $G$.
    Then $f \sim_k g$  if and only if $\xb_f- \xb_g \in \langle [I_G]_2 \rangle$. 
\end{Proposition}

We will see a proof of this proposition in Section \ref{sec:kempe} (Theorem \ref{ExamKempe})
by using a similar discussion in the proof of Proposition \ref{prop:quad} in \cite{OTkempe}.

\section{Gr\"{o}bner bases}
\label{sect:GB}
In this section, we give a brief introduction to Gr\"{o}bner bases.
Let $R=\KK[x_1,x_2,\ldots,x_n]$ be a polynomial ring over a field $\KK$ with $\deg(x_i)=1$, and denote $\mathcal{M}_n$ the set of all monomials in the variables $x_1,\ldots,x_n$.
A \textit{monomial order} on $R$ is a total order $<$ on $\mathcal{M}_n$ such that
\begin{enumerate}
    \item $1< u$ for all $1 \neq u \in \mathcal{M}_n$;
    \item if $u,v \in \mathcal{M}_n$ and $u < v$, then $uw < vw$ for all $w \in \mathcal{M}_n$.
\end{enumerate}
We give an example of a monomial order. 

\begin{Example}
Let $u=x_1^{a_1} \cdots x_n^{a_n}$ and $v=x_1^{b_1} \cdots x_n^{b_n}$ be monomials in $\mathcal{M}_n$. We define the total order $<_{\rm rev}$ on $\mathcal{M}_n$ by setting $u <_{\rm rev} v$ if either (i) $\sum_{i=1}^n a_i < \sum_{i=1}^n b_i$, or 
(ii) $\sum_{i=1}^n a_i = \sum_{i=1}^n b_i$ and the rightmost nonzero component of the vector $(b_1-a_1, b_2-a_2,\ldots,b_n-a_n)$ is negative. It then follows that $<_{\rm rev}$ is a monomial order on $R$, which is called the (\textit{graded}) \textit{reverse lexicographic order} on $R$ induced by the ordering $x_1 > x_2 > \cdots > x_n$.
By reordering the variables, we can obtain another reverse lexicographic order on $R$.
Hence, there are $n!$ reverse lexicographic orders on $R$.
\end{Example}

Fix a monomial order $<$ on $R$.
For a nonzero polynomial $f$ of $R$, the \textit{support} of $f$, denoted by ${\rm supp}(f)$, is the set of all monomials appearing in $f$ and the \textit{initial monomial} ${\rm in}_<(f)$ of $f$ with respect to $<$ is the largest monomial belonging to ${\rm supp}(f)$ with respect to $<$.
Let $I$ be a nonzero ideal of $R$. Then the \textit{initial ideal} ${\rm in}_<(I)$
of $I$ with respect to $<$ is defined as follows:
\[
{\rm in}_<(I)= \langle {\rm in}_<(f) \mid 0 \neq f \in I \rangle \subset R.
\]
In general, even if $I=\langle f_1,\ldots,f_s \rangle$, it is not necessarily true that 
\[{\rm in}_<(I)=\langle {\rm in}_<(f_1),\ldots,{\rm in}_<(f_s) \rangle.
\]
A finite set $\{g_1,\ldots,g_s\}$ of nonzero polynomials belonging to $I$ is called a \textit{Gr\"{o}bner basis} of $I$ with respect to $<$ if ${\rm in}_<(I)=\langle {\rm in}_<(g_1),\ldots,{\rm in}_<(g_s) \rangle$.
Note that if $\{g_1,\ldots,g_s\}$ is a Gr\"{o}bner basis of $I$, then $I=\langle g_1,\ldots,g_s \rangle$.
It is known that $I$ always has a Gr\"{o}bner basis.
A Gr\"{o}bner basis $\Gc$ for $I$ with respect to $<$ is called \textit{reduced} if the following conditions hold:
\begin{enumerate}
    \item for all $p \in \Gc$, the coefficient of ${\rm in}_<(p)$ in $p$ equals $1$;
    \item for all $p \in \Gc$, no monomial of $p$ lies in $\langle  {\rm in}_<(g) \mid g \in \Gc \setminus \{p\}\rangle$.
\end{enumerate}
 A nonzero ideal $I$ has a unique reduced Gr\"{o}bner basis with respect to a fixed monomial order $<$. In particular, given a Gr\"{o}bner basis, we can easily get the reduced Gr\"{o}bner basis from it.

Next we introduce a method for determining whether a finite system of generators of $I$ is a Gr\"{o}bner basis of $I$ with respect to $<$.
For two nonzero polynomials $f$ and $g$ in $R$, the polynomial 
\[
S(f,g)=\frac{m}{c_f \cdot {\rm in}_< (f)} \cdot f -\frac{m}{c_g \cdot {\rm in}_< (g)} \cdot g
\]
is called the \textit{$S$-polynomial} of $f$ and $g$, where 
$c_f$ is the coefficient of ${\rm in}_<(f)$ in $f$ and $c_g$ is that of ${\rm in}_<(g)$ in $g$, and $m$ is the least common multiple of the initial monomials ${\rm in}_<(f)$ and ${\rm in}_<(g)$.

\begin{Lemma}[Buchberger's Criterion {\cite[Chapter 2, \S 9, Theorem 3]{CLO}}]
    Let $I$ be a nonzero ideal of $R$, $\mathcal{G}=\{g_1,\ldots,g_s\}$ a finite system of generators of $I$.
    Then $\Gc$ is a Gr\"{o}bner basis of $I$ with respect to a monomial order $<$ on $R$ if and only if the remainder of $S$-polynomial $S(g_i,g_j)$ on division by $\Gc$ is $0$ for all $i \neq j$.
\end{Lemma}

The algorithm called 
\textit{Buchberger's Algorithm} (\cite[Chapter 2, \S 7, Theorem 2]{CLO}) is based on Buchberger's Criterion,
and computes a Gr\"{o}bner basis for $I$ from a finite system of generators of $I$.

As applications of Gr\"{o}bner bases, we can determine if a polynomial $f$ belongs to an ideal. This problem is called the \textit{ideal membership problem}. 
If $\Gc$ is a Gr\"{o}bner basis for an ideal $I$ of $R$ with respect to a monomial order $<$, then every polynomial of $f \in R$ has a unique remainder on division by $\Gc$.
The remainder is called the \textit{normal form} of $f$ with respect to $\Gc$.

\begin{Lemma}[{\cite[Chapter 2, \S 6, Corollary 2]{CLO}}] \label{IMP}
     Let $I$ be a nonzero ideal of $R$ and $\mathcal{G}$ a Gr\"{o}bner basis of $I$ with respect to a monomial order $<$ on $R$. Fix a polynomial $f$ in $R$.
Then $f \in I$ if and only if the normal form of $f$ with respect to $\Gc$ equals $0$.
     \end{Lemma}

Next, we review a method for computing Hilbert functions.
Let $I$ be a graded ideal of $R$. The numerical function $H(R/I,-): \ZZ_{\geq 0} \to \ZZ_{\geq 0}$ with $H(R/I,k)=\dim_{\KK} R_k/I_k$ is called the \textit{Hilbert function} of $R/I$, where $R_k$ (resp. $I_k$) is the homogeneous part of degree $k$ of $R$ (resp. $I$) and $\dim_{\KK} R_k/I_k$ is the dimension of $R_k/I_k$ as $\KK$-vector spaces.
For a monomial order $<$ on $R$, a monomial $u \in \mathcal{M}_n$ is said to be \textit{standard} with respect to ${\rm in}_<(I)$ if $u \notin {\rm in}_<(I)$.
\begin{Lemma}[{\cite[Theorem 1.19]{HHO}}]
Let $I$ be a nonzero graded ideal of $R$ and fix a monomial order $<$ on $R$. Let $\mathcal{B}_k$ denote the set of standard monomials of degree $k$ with respect to ${\rm in}_<(I)$.
Then $\mathcal{B}_k$ is a $\KK$-basis of $R_k/I_k$ as a $\KK$-vector space. 
In particular, one has
\[
H(R/I,k)=|\mathcal{B}_k|.
\]
\end{Lemma}

Finally, we recall a property of the saturation of an ideal.
For a nonzero ideal $I$ of $R$ and a polynomial $f \in R$, the \textit{saturation} of $I$  with respect to $f$ is the ideal
\[
I : f^{\infty}=\{g \in R \mid \mbox{there exists $i>0$ such that $f^ig \in I$}\}.
\]

\begin{Lemma}[{\cite[Proposition 1.40]{HHO}}] \label{colonGB}
    Let $I$ be a nonzero graded ideal of $R$ and let $\mathcal{G}$ be the reduced Gr\"{o}bner basis of $I$ with respect to the reverse lexicographic order induced by $x_1 > x_2 > \cdots > x_n$.
    Then
    \[
\{ g /x_n^k \mid g \in \mathcal{G}, \ k \in \ZZ_{\geq 0},  \ 
x_n^k \mbox{ divides } g,\ x_n^{k+1} \mbox{ does not divide } g \} 
    \]
    is a Gr\"{o}bner basis of \ $I : x_n^{\infty}$.
\end{Lemma}

\section{2-coloring ideals}
\label{sect:2-color}
Given a graph $G$ on $[d]$, we define the following two ideals of $R[G]$:
\begin{align*}
    J_G:=& \langle \xb_f - \xb_g \mid \mbox{$f$ and $g$ are $2$-colorings of $G_{\ab}$ with $\ab \in \{0,1\}^d$}\rangle \\
    =&\langle \xb_f - \xb_g \mid \mbox{$f$ and $g$ are $2$-colorings of an induced subgraph of $G$} \rangle,\\
    L_G:=&\langle x_{S \setminus \{i\}} x_{\{i\}}- x_{S} x_\emptyset \mid i \in S \in S(G), \ |S| \ge 2\rangle. 
\end{align*}
We call $J_G$ the \textit{$2$-coloring ideal} of $G$.
Note that inclusions of ideals
\begin{equation}
\label{ineq}
L_G \subset J_G \subset  \langle [I_G]_2 \rangle \subset I_G    
\end{equation}
hold.
Moreover, from Proposition~\ref{prop:gen}, for $I \in \{ \langle [I_G]_2\rangle, J_G, L_G\}$, if $\xb_f-\xb_g \in I$, then $f$ and $g$ are $k$-colorings of $G_{\ab}$ with $\ab \in \ZZ_{\geq 0}^d$.

In this section, we discuss some properties of these ideals.
\begin{Lemma}
\label{colongraph}
    Let $G$ be a graph on $[d]$.
    Then one has 
 $L_G : \xb_{\emptyset}^{\infty}=J_G : \xb_{\emptyset}^{\infty}=\langle [I_G]_2 \rangle : \xb_{\emptyset}^{\infty} =I_G$.
  \end{Lemma}
\begin{proof}
Since $L_G \subset J_G \subset  \langle [I_G]_2 \rangle \subset I_G$, we have 
$
L_G : \xb_{\emptyset}^{\infty}
\subset 
J_G : \xb_{\emptyset}^{\infty}
\subset
\langle [I_G]_2 \rangle : \xb_{\emptyset}^{\infty} 
\subset 
I_G: \xb_{\emptyset}^{\infty}.
$
In addition, since $I_G$ is prime and does not contain $x_\emptyset^k$ for any $k$, we have 
$I_G: \xb_{\emptyset}^{\infty} = I_G$.
Thus it is sufficient to show that 
$I_G \subset  L_G : \xb_{\emptyset}^{\infty}$.
Let 
\[
F = \prod_{i=1}^s  x_{S_i} - \prod_{i=1}^s  x_{S_i'} \in I_G,
\]
where $S_i = \{k_1^{(i)} ,\dots, k_{p_i}^{(i)}\}$, $S_i'= \{\ell_1^{(i)} ,\dots, \ell_{q_i}^{(i)}\}\in S(G)$ for each $i$.
Then
\begin{eqnarray*}
 & & x_{S_i} x_\emptyset^{p_i-1} - \prod_{j=1}^{p_i} x_{\{k_j^{(i)}\}}\\
&=&
x_\emptyset^{p_i-2} \left( x_{S_i} x_\emptyset- x_{S_i \setminus \{k_1^{(i)}\}} x_{\{k_1^{(i)}\}}\right) 
+
x_\emptyset^{p_i-3} x_{\{k_1^{(i)}\}} \left( x_{S_i \setminus \{k_1^{(i)}\}} x_\emptyset- x_{S_i \setminus \{k_1^{(i)}, k_2^{(i)}\}} x_{\{k_2^{(i)}\}}\right) \\
 & & + \cdots + 
\left(\prod_{j=1}^{p_i-2} x_{\{k_j^{(i)}\}}\right)
\left( x_{\{k_{p_i -1}^{(i)},k_{p_i}^{(i)} \}} x_\emptyset- x_{\{k_{p_i}^{(i)}\}} x_{\{k_{p_i -1}^{(i)} \}} \right)
\end{eqnarray*}
belongs to $L_G$.
Note that if binomials $u_1 - v_1,\dots,u_s - v_s$ belong to $L_G$, then 
\begin{eqnarray*}
u_1 \dots u_s - v_1 \dots v_s
&=& u_2 \cdots u_s (u_1 - v_1) + u_3 \cdots u_s v_1 (u_2 - v_2)  +u_4 \cdots u_s v_1 v_2 (u_3 - v_3)\\
&&
+ \cdots + v_1 \cdots v_{s-1} (u_s - v_s) 
\end{eqnarray*}
belongs to $L_G$.
Hence 
\[
G_1=
x_\emptyset^{\sum_{i=1}^s (p_i-1)}
\prod_{i=1}^s  x_{S_i} -
\prod_{i=1}^{s}\prod_{j=1}^{p_i} x_{\{k_j^{(i)}\}}
\]
belongs to $L_G$.
By the same argument, 
\[
G_2=
x_\emptyset^{\sum_{i=1}^s (q_i-1)}
\prod_{i=1}^s  x_{S_i'} -
\prod_{i=1}^{s}\prod_{j=1}^{q_i} x_{\{\ell_j^{(i)}\}}
\]
belongs to $L_G$.
Since $F$ belongs to $I_G$,
$\pi(\prod_{i=1}^s  x_{S_i})=\pi(\prod_{i=1}^s  x_{S_i'})$.
This implies that $\bigcup_{i=1}^s S_i=\bigcup_{i=1}^s S'_i$ as multisets. 
Hence one has $\sum_{i=1}^s p_i=\sum_{i=1}^s q_i$
and
\[
\prod_{i=1}^{s}\prod_{j=1}^{p_i} x_{\{k_j^{(i)}\}}=
\prod_{i=1}^{s}\prod_{j=1}^{q_i} x_{\{\ell_j^{(i)}\}}.
\]
Thus 
\[
x_\emptyset^{\sum_{i=1}^s (p_i-1)}F=
G_1-G_2
\]
belongs to $L_G$.
This implies that $F \in L_G : x^{\infty}_\emptyset$. 
Hence one has $I_G \subset L_G : x^{\infty}_\emptyset$.
\end{proof}
Therefore, combining Lemmata \ref{colonGB} and \ref{colongraph}, we obtain the following proposition.

\begin{Proposition} \label{I to I_G}
Let $G$ be a graph on $[d]$ and 
let $I \in \{\langle [I_G]_2 \rangle, J_G, L_G\}$.
If $\mathcal{G}$ is the reduced Gr\"{o}bner basis of $I$ with respect to a reverse lexicographic order such that $x_S  \geq x_{\emptyset}$ for any $S \in S(G)$, 
then
            \[
\{ g /x_{\emptyset}^k \mid g \in \mathcal{G}, k \in \ZZ_{\geq 0}, \mbox{$x_{\emptyset}^k$ divides $g$, $x_{\emptyset}^{k+1}$ does not divide $g$} \} 
    \]
    is a Gr\"{o}bner basis of $I_G$.
\end{Proposition}

Next, we discuss 
when equality holds in the inclusions $L_G \subset J_G \subset \langle [I_G]_2 \rangle \subset I_G$.
In particular, we characterize when these ideals are prime.
Let $\overline{G}$ denote the complement graph of a graph $G$.

\begin{Proposition}
\label{prime}
Let $G$ be a graph on $[d]$. Then one has the following.
\begin{enumerate}
    \item[{\rm (a)}] 
    Let $I \in \{ \langle [I_G]_2 \rangle, J_G, L_G\}$. Then $I$ is prime $\iff$ $I=I_G$.

    \item[{\rm (b)}]
    $I_G=\langle [I_G]_2 \rangle \iff I_G$ is generated by quadratic binomials.

    \item[{\rm (c)}]
    $\langle [I_G]_2 \rangle =J_G \iff \overline{G}$ has no $3$-cycles.
    
    \item[{\rm (d)}]
    $J_G=L_G \iff 
    G$ is a complete multipartite graph on the vertex set
    $V_1 \sqcup \cdots \sqcup V_t$ with $|V_j| \le 3$ for each $j$.
\end{enumerate}
\end{Proposition}
\begin{proof} 
(a) Let $I \in \{ \langle [I_G]_2 \rangle, J_G, L_G\}$.
Recall that the toric ideal $I_G$ is prime.
Hence $I$ is prime if $I_G = I$.
Conversely, suppose that $I$ is prime.
Since $x^k_{\emptyset} \notin I$ holds for any $k$, it follows that
$I=I : x_{\emptyset}^{\infty}$.
From Lemma \ref{colongraph}, we have
$I=I : x_{\emptyset}^{\infty} =I_G$.

(b) Trivial.

(c)
If $\overline{G}$ has a $3$-cycle $(i_1,i_2,i_3)$, then
$x_{\{i_1,i_2,i_3\}} x_{\{i_1\}} -x_{\{i_1,i_2\}} x_{\{i_1,i_3\}} 
\in \langle [I_G]_2 \rangle \setminus J_G$.
Suppose that 
$\overline{G}$ has no $3$-cycles.
Then $|S| \le 2$ for all $S \in S(G)$.
Let $h= x_{S_1} x_{S_2} - x_{S_3} x_{S_4} \in \langle [I_G]_2 \rangle$.
Since $h$ belongs to $I_G$, $S_1 \cup S_2$ coincides with $S_3 \cup S_4$ as multisets, and in particular,
we have $S_1 \cap S_2 = S_3 \cap S_4$.
Suppose that $S_1 \cap S_2$ is not empty.
Let $i \in S_1 \cap S_2$.
Then $(S_1 \setminus \{i\}) \cup (S_2 \setminus \{i\} )$ coincides with $(S_3 \setminus \{i\}) \cup (S_4 \setminus \{i\} )$
as multisets.
Since $|S_j \setminus \{i\}| \le 1$ for each $j = 1,2,3,4$, it follows that 
$(S_1,S_2)$ is equal to either $(S_3,S_4)$ or $(S_4,S_3)$.
Then $h = 0$, a contradiction.
Thus $S_1 \cap S_2 = \emptyset$, and hence $h \in J_G$.
Therefore we have $\langle [I_G]_2 \rangle \subset J_G$, and hence $\langle [I_G]_2 \rangle = J_G$.

(d)
Suppose that $J_G = L_G$.
If $\overline{G}$ has a path of length two $P_3 = (i,j,k)$ as an induced subgraph, then it follows that 
$x_{\{i,j\}} x_{\{k\}} -x_{\{i\}} x_{\{j,k\}} 
\in J_G \setminus L_G $, a contradiction.
Hence $\overline{G}$ has no $P_3$ as an induced subgraph.
It is known that $\overline{G}$ has no $P_3$ as an induced subgraph if and only if $G$ is a complete multipartite graph.
Suppose that $G$ has a part $V_\alpha$ with $|V_\alpha|\ge 4$.
It then follows that $x_{\{i,j\}} x_{\{k,l\}} -x_{\{i,k\}} x_{\{j,l\}} 
\in J_G \setminus L_G $ where $i,j,k,l$ are distinct vertices in $V_\alpha$, a contradiction.
Hence $G$ is a complete multipartite graph on the vertex set
$V_1 \sqcup \cdots \sqcup V_t$ with $|V_j| \le 3$ for each $j$.

Suppose that $G$ is a complete multipartite graph on the vertex set
$V_1 \sqcup \cdots \sqcup V_t$ with $|V_j| \le 3$ for each $j$.
Then each $S \in S(G)$ is a subset of $V_j$ for some $j$.
It is enough to show that $J_G \subset L_G$.
Let $h= x_{S_1} x_{S_2} - x_{S_3} x_{S_4}$ be a nonzero binomial in $J_G$.
Then $S_1 \cap S_2 = S_3 \cap S_4 =\emptyset$ and $S_1 \cup S_2 = S_3 \cup S_4$.
Suppose that $S_i \subset V_{j_i}$ for each $i = 1,2,3,4$.

\bigskip

\noindent
{\bf Case 1.}
($S_1 = \emptyset$.)
Then $h = x_\emptyset x_{S_2} - x_{S_3} x_{S_4} $, where $|S_3|, |S_4| \ge 1$ and $ |S_3| + |S_4|=|S_2| \le 3$.
Since either $|S_3|$ or $|S_4|$ equals one, $h \in L_G$.

\bigskip

\noindent
{\bf Case 2.}
($S_j \ne \emptyset$ for each $i$.)
Since $S_1 \cup S_2 = S_3 \cup S_4$, we may assume that $j_1 = j_3$ and $j_2 = j_4$.
If $j_1 \ne j_2$, then $S_1 = S_3$ and $S_2 = S_4$, and hence $h=0$.
This is a contradiction.
Thus we have $j_1 = j_2 = j_3 = j_4$.
%
Since $2 \le |S_1| + |S_2| = |S_3| + |S_4| \le 3$, we may assume that $|S_1| = |S_3|=1$.
Then $h = (x_{S_1} x_{S_2} - x_\emptyset x_{S_1 \cup S_2}) + (x_\emptyset x_{S_3 \cup S_4} - x_{S_3} x_{S_4})$ belongs to $L_G$.
\end{proof}

Now, we return to the inclusions (\ref{ineq}).
If $G$ is a complete multipartite graph, then $I_G$ is generated by quadratic binomials (\cite[Theorem 3.1]{OST}).  Hence from Proposition \ref{prime} we know that $L_G=J_G$ implies $\langle [I_G]_2 \rangle = I_G$.
Therefore, there are the following $6$ cases:
\begin{enumerate}[(i)]
    \item $L_G = J_G = \langle [I_G]_2 \rangle = I_G$.
     \item $L_G = J_G \subsetneq \langle [I_G]_2 \rangle = I_G$.
    \item $L_G \subsetneq J_G = \langle [I_G]_2 \rangle = I_G$.
    \item $L_G \subsetneq J_G = \langle [I_G]_2 \rangle \subsetneq I_G$.
    \item $L_G \subsetneq J_G \subsetneq \langle [I_G]_2 \rangle = I_G$.
    \item $L_G \subsetneq J_G \subsetneq \langle [I_G]_2 \rangle \subsetneq I_G$.
\end{enumerate}
We give an example for each case.
\begin{Example}
First, we recall that for any bipartite graph $G$, $I_G$ is generated by quadratic binomials.

A complete graph $K_d$ of $d$ vertices satisfies the condition (i). In fact, one has $I_{K_d}=\{0\}$.
For a $(3,3)$-complete bipartite graph $K_{3,3}$, since $\overline{K_{3,3}}$ has a $3$-cycle, the condition (ii) holds.
For a path of length three $P_4$, since $\overline{P_4}$ is also a path of length three, the condition (iii) holds.
For $G=\overline{C_6}$, $I_{G}$ is not generated by quadratic binomials from \cite[Proposition 11]{MOS}. Since $\overline{G}=C_6$ has no $3$-cycles and $G$ is not a complete multipartite graph, the condition (iv) is satisfied.
For a $6$-cycle $C_6$, 
since $\overline{C_6}$ has a $3$-cycle and $C_6$ is not a complete multipartite graph,
the condition (v) holds.
Finally, we consider the following graph $G$:
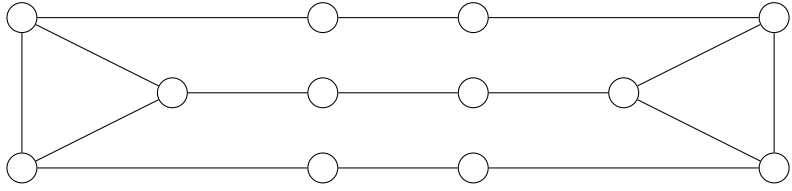
	\begin{center}
		\begin{tikzpicture}
\node[draw, shape=circle] (i1) at (0,2) {}; 
\node[draw, shape=circle] (k1) at (2,1) {}; 
\node[draw, shape=circle] (j1) at (0,0) {}; 
\node[draw, shape=circle] (i2s) at (10,2) {}; 
\node[draw, shape=circle] (k2u) at (8,1) {}; 
\node[draw, shape=circle] (j2t) at (10,0) {}; 
\node[draw, shape=circle] (i2) at (4,2) {};
\node[draw, shape=circle] (i2s1) at (6,2) {}; 
\node[draw, shape=circle] (j2) at (4,0) {};
\node[draw, shape=circle] (j2t1) at (6,0) {}; 
 \node[draw, shape=circle] (k2) at (4,1) {};
\node[draw, shape=circle] (k2u1) at (6,1) {};

\draw (i1)--(i2);
\draw (j1)--(j2);
\draw (k1)--(k2);
\draw (i1)--(j1);
\draw (i1)--(k1);
\draw (j1)--(k1);
\draw (i2s)--(j2t);
\draw (i2s)--(k2u);
\draw (j2t)--(k2u);
\draw (i2s1)--(i2s);
\draw (j2t1)--(j2t);
\draw (k2u1)--(k2u);

\draw (i2s1)--(i2);
\draw (j2t1)--(j2);
\draw (k2u1)--(k2);
\end{tikzpicture}
	\end{center}
 It then follows from \cite[Theorem 1.7]{OST} that $I_G$ is not generated by quadratic binomials. Since $\overline{G}$ has a $3$-cycle and $G$ is not a complete multipartite graph, the condition (vi) holds.
\end{Example}

\section{Examining Kempe equivalence}
\label{sec:kempe}
In this section, we prove Theorem \ref{thm:K_equiv} and Proposition \ref{prop:exam_kempe}.
In fact, we show the following.

\begin{Theorem} \label{ExamKempe}
    Let $G$ be a graph on $[d]$ and let $f$ and $g$ be $k$-colorings of an induced subgraph of $G$.
    Then the following conditions are equivalent{\rm :}
    \begin{enumerate}
        \item[{\rm (i)}] $f \sim_k g${\rm ;}
        \item[{\rm (ii)}] $\xb_f - \xb_g \in \langle [I_G]_2 \rangle${\rm ;}
        \item[{\rm (iii)}] $\xb_f - \xb_g \in J_G$.
    \end{enumerate}
\end{Theorem}

\begin{proof}
Suppose that $f$ and $g$ are $k$-colorings of an induced subgraph $G_0$ of $G$.
Since $J_G \subset  \langle [I_G]_2 \rangle $ holds, we have (iii) $\Rightarrow$ (ii).

(i) $\Rightarrow$ (iii).
Suppose that $f \sim_k g$
and $\xb_f - \xb_g \notin J_G$.
Let $f_0,f_1,\ldots,f_s$ be a sequence of $k$-colorings of $G_0$ such that $f_0=f$, $f_s=g$,
and $f_i$ is obtained from $f_{i-1}$ by a Kempe switching.
We may assume that $s \ge 1$ is minimal among 
all $k$-colorings $f$ and $g$ such that
$f \sim_k g$
and $\xb_f - \xb_g \notin J_G$.
Suppose that the Kempe switching from $f$ to $f_1$ is obtained by
a connected component $H$ of the induced subgraph 
$G_0 [f^{-1}(\mu) \cup f^{-1}(\eta)]$ by setting
$$
f_1(x) =
\begin{cases}
    f(x) & x \notin H,\\
    \mu & x \in H \mbox{ and } f(x) =\eta,\\
    \eta & x \in H \mbox{ and } f(x) =\mu.\\
\end{cases}
$$
Let $f' = f|_{G'}$ and $f_1' = f_1|_{G'}$ be the restrictions of  $f$ and $f_1$ to $G':= G_0  [f^{-1}(\mu) \cup f^{-1}(\eta)]$, respectively.
Since $f'$ and $f_1'$ are 2-colorings of ${G'}$, 
$\xb_{f'} - \xb_{f_1'}$ belongs to $J_G$.
Then 
$$
\xb_f -\xb_g = \frac{\xb_f}{\xb_{f'}} 
(\xb_{f'} - \xb_{f_1'} )
+ \left(\frac{\xb_f}{\xb_{f'}} 
 \xb_{f_1'} - \xb_g\right)
 =
 \frac{\xb_f}{\xb_{f'}} 
(\xb_{f'} - \xb_{f_1'} )+
\xb_{f_1} - \xb_{g}
. $$ 
If $s=1$, then $f_1 = g$ and hence 
$$
\xb_f -\xb_g = 
 \frac{\xb_f}{\xb_{f'}} 
(\xb_{f'} - \xb_{f_1'} ) \in J_G
. $$ 
We may assume that $s \ge 2$.
By the hypothesis on $s$, $\xb_{f_1} - \xb_{g}$ belongs to $J_G$.
Hence $\xb_f -\xb_g$ belongs to $J_G$, a contradiction.

(ii) $\Rightarrow$ (i).
Suppose that $f \not\sim_k g$
and $\xb_f - \xb_g \in \langle [I_G]_2 \rangle$.
Since $\langle [I_G]_2 \rangle$ is a binomial ideal, it then follows from \cite[Lemma 3.8]{HHO} that
there exists an expression
\begin{equation}
\xb_f - \xb_g  = \sum_{r=1}^s {\bf x}_{w_r} (\xb_{f_r}- \xb_{g_r}),
\label{tenkai}
\end{equation}
where 
$f_r$ and $g_r$ are 2-colorings of $G_{\ab_r}$ with $\ab_r \in \ZZ_{\ge 0}^d$ for each $r$.
We may assume that $s \ge 1$ is minimal among such $f$ and $g$
with 
$f \not\sim_k g$
and $\xb_f - \xb_g \in \langle [I_G]_2 \rangle$.
Since $\xb_f$ must appear in the right-hand side of (\ref{tenkai}), we may assume that
$\xb_f = \xb_{w_1} \xb_{f_1}$.
Then $\xb_f$ is divided by $\xb_{f_1}$,
and $f_1$ is the restriction of $f$
to $G' = G_0 [f^{-1} (\mu) \cup f^{-1} (\nu)]$
for some $\mu$ and $\nu$.
Since $G'$ has a 2-coloring, it is a bipartite graph.
Then $f_1$ and $g_1$ are Kempe equivalent.
Let $f'$ be a coloring of $G_0$ defined by
$$
f'(x) =
\begin{cases}
    g_1(x) &  f(x) \in \{\mu, \nu\},\\
         f(x) &\mbox{otherwise.}
\end{cases}
$$
Then one obtains $f \sim_k f'$.
Moreover, we have
\begin{equation}
\xb_f -{\bf x}_{w_1} (\xb_{f_1}- \xb_{g_1}) = \xb_{f'}.
\end{equation}
If $s=1$, then 
\[
\xb_g = 
\xb_f -{\bf x}_{w_1} (\xb_{f_1}- \xb_{g_1}) = \xb_{f'},
\]
and hence $f \sim_k  f' = g$, a contradiction.
We may assume that $s\ge 2$.
Since
\[
\xb_{f'} - \xb_g
=
\xb_f -{\bf x}_{w_1} (\xb_{f_1}- \xb_{g_1})  - \xb_g
=
\sum_{r=2}^s {\bf x}_{w_r} (\xb_{f_r}- \xb_{g_r}) \in \langle [I_G]_2 \rangle,
\]
$f' \sim_k g$ by the hypothesis on $s$.
Thus $f \sim_k f' \sim_k g$, a contradiction.
\end{proof}

We see examples of Theorem \ref{ExamKempe}.
\begin{Example}
    {\rm 
    Let $G$ be the graph as follows:
    			
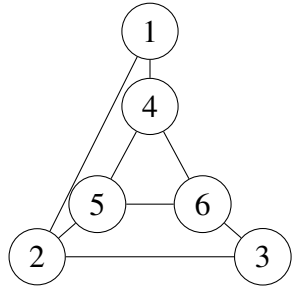
\begin{center}
		\begin{tikzpicture}
\node[draw, shape=circle] (1) at (0,3) {$1$};	
\node[draw, shape=circle] (2) at (-1.5,0) {$2$}; 
\node[draw, shape=circle] (3) at (1.5,0) {$3$}; 
\node[draw, shape=circle] (4) at (0,2) {$4$}; 
\node[draw, shape=circle] (5) at (-0.7,0.7) {$5$}; 
\node[draw, shape=circle] (6) at (0.7,0.7) {$6$}; 

\draw (1)--(2);
\draw (2)--(3);
\draw (1)--(4);
\draw (2)--(5);
\draw (3)--(6);
\draw (4)--(5);
\draw (6)--(4);
\draw (5)--(6);
\end{tikzpicture}
	\end{center}
We consider two $3$-colorings $f$ and $g$ of $G$ defined by
\[
f(i)=\begin{cases}
    1 & i \in \{1,5\},\\
    2 & i \in \{2,6\},\\
    3 & i \in \{3,4\}
\end{cases}
\mbox{ and }
g(i)=\begin{cases}
    1 & i \in \{1,3,5\},\\
    2 & i \in \{2,6\},\\
    3 & i =4.
\end{cases}
\]
Then one has \[\xb_f-\xb_g=x_{\{1,5\}} x_{\{2,6\}} x_{\{3,4\}}-x_{\{1,3,5\}} x_{\{2,6\}} x_{\{4\}}=x_{\{2,6\}} (x_{\{1,5\}}  x_{\{3,4\}}-x_{\{1,3,5\}} x_{\{4\}}).\]
    Since $x_{\{1,5\}} x_{\{3,4\}}-x_{\{1,3,5\}} x_{\{4\}} \in J_G$, it then follows from Theorem \ref{ExamKempe} that $f \sim_3 g$.
   }
\end{Example}
\begin{Example}
\label{ex:nonkempe}
    {\rm 
    Let $G$ be the graph as follows:
    			
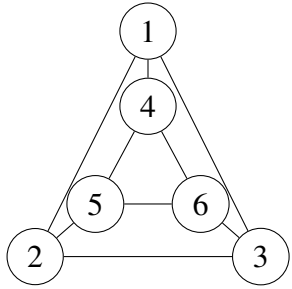
\begin{center}
		\begin{tikzpicture}
\node[draw, shape=circle] (1) at (0,3) {$1$};	
\node[draw, shape=circle] (2) at (-1.5,0) {$2$}; 
\node[draw, shape=circle] (3) at (1.5,0) {$3$}; 
\node[draw, shape=circle] (4) at (0,2) {$4$}; 
\node[draw, shape=circle] (5) at (-0.7,0.7) {$5$}; 
\node[draw, shape=circle] (6) at (0.7,0.7) {$6$}; 

\draw (1)--(2);
\draw (2)--(3);
\draw (1)--(3);
\draw (1)--(4);
\draw (2)--(5);
\draw (3)--(6);
\draw (4)--(5);
\draw (6)--(4);
\draw (5)--(6);
\end{tikzpicture}
	\end{center}
We consider two $3$-colorings $f$ and $g$ of $G$ defined by
\[
f(i)=\begin{cases}
    1 & i \in \{1,5\},\\
    2 & i \in \{2,6\},\\
    3 & i \in \{3,4\}
\end{cases}
\mbox{ and }
g(i)=\begin{cases}
    1 & i \in \{1,6\},\\
    2 & i \in \{2,4\},\\
    3 & i \in \{3,5\}.
\end{cases}
\]
Then one has \[\xb_f-\xb_g=x_{\{1,5\}} x_{\{2,6\}} x_{\{3,4\}}-x_{\{1,6\}} x_{\{2,4\}} x_{\{3,5\}}.\]
    Since $\xb_f-\xb_g \notin J_G$, it then follows from Theorem \ref{ExamKempe} that $f \not\sim_3 g$.
    }
\end{Example}

 See Section \ref{sect:algo} for an algorithm to  determine if $f \sim_k g$ by using the techniques on Gr\"{o}bner bases.

\section{Computing Kempe classes via Kempe ideals}
\label{sect:k-class}
In this section, we give a way to find all $k$-colorings of a graph $G$ up to Kempe equivalence via commutative algebra.
In particular, we prove Theorem \ref{thm:K-class}.
For a graph $G$ on $[d]$, we define the ideal $K_G$ of $R$ by
\[
K_G:=J_G + M_G \subset R[G],
\]
where 
\[
M_G := \langle x_Sx_T \mid S,T \in S(G), \ S \cap T \neq \emptyset \rangle \subset R[G].
\]
The ideal $K_G$ is called the \textit{Kempe ideal} of $G$.

The set of all monomials in $M_G$ consists of all monomials associated with $k$-colorings of $G_{\ab}$ with $\ab \in \ZZ_{\geq 0}^d$ such that $G_{\ab}$ is not an induced subgraph of $G$.  
\begin{Lemma} \label{not in M}
Let $G$ be a graph on $[d]$.
Then $\xb_f =x_{S_1} x_{S_2} \cdots x_{S_k} \notin M_G$ if and only if $f$ is a $k$-coloring of an induced subgraph of $G$. 
\end{Lemma}

\begin{proof}
Suppose that $f$ is not a $k$-coloring of any induced subgraph of $G$.
Then $f$ is a $k$-coloring of $G_{\ab}$ with $\ab \in \ZZ_{\geq 0}^d$ such that $a_j \geq 2$ for some $j$. Hence there exist integers $1 \leq a < b \leq k$ with $i \in S_a$ and $i \in S_b$. 
This implies that $\xb_f$ can be divided by a monomial $x_{S_a} x_{S_b} \in M_G$,
and hence $\xb_f \in M_G$.

Suppose that $\xb_f \in M_G$.
There exist $S,T \in S(G)$ such that $S \cap T \neq \emptyset$ and $x_S x_T$ divides $\xb_f$.
Let $i \in S \cap T$.
Then $i$ appears at least twice in $S_1 \cup \dots \cup S_k$.
Thus $f$ is not a $k$-coloring of any induced subgraph of $G$.
\end{proof}

As an easy consequence of this lemma, we have the following.
\begin{Corollary}
    Let $G$ be a graph on $[d]$. Then one has
    \[
    K_G=\langle [I_G]_2 \rangle + M_G.
    \]
\end{Corollary}

The following lemma is useful to check if a homogeneous binomial in $I_G$ belongs to $M_G$.

\begin{Lemma} \label{belongs to M}
Let $G$ be a graph on $[d]$ and let  $\xb_f - \xb_g \in I_G$.
Then the following conditions are equivalent{\rm :}
\begin{itemize}
    \item[{\rm (i)}]
    $\xb_f - \xb_g \in M_G${\rm ;}
\item[{\rm (ii)}]
both $\xb_f$ and $\xb_g$ belong to $M_G${\rm ;}
\item[{\rm (iii)}]
at least one of $\xb_f$ and $\xb_g$ belongs to $M_G$.
    \end{itemize}
\end{Lemma}

\begin{proof}
First, (ii) $\Longrightarrow$ (iii) is trivial.
Since $M_G$ is a monomial ideal, the equivalence (i) $\iff$ (ii) holds.
We show that (iii) $\Longrightarrow$ (ii).
    Suppose that $\xb_f \in M_G$, i.e., $\xb_f$  is divided by $x_S x_T$ where $S, T \in S(G)$ and $S \cap T \neq \emptyset$.
    Let $i \in S \cap T$.
    Then $\pi(\xb_f)$ is divided by $t_i^2$ since $\xb_f$ is divided by $x_S x_T$.
    From $\xb_f - \xb_g \in I_G$, we have $\pi(\xb_f)=\pi(\xb_g)$.
    Hence if $\xb_g=x_{S_1}x_{S_2}\cdots x_{S_k}$, then there exist $1 \leq a < b \leq k$ such that $i \in S_a\cap S_b$. Namely,
    $\xb_g$ is divided by a monomial $x_{S_a} x_{S_b}$ such that $i \in S_a \cap S_b$.
    Thus $\xb_g \in M_G$.
\end{proof}

Next, we discuss a Gr\"{o}bner basis of $K_G$.

\begin{Proposition} \label{kg GB}
Let $G$ be a graph on $[d]$, and
    let $\mathcal{G}_1$ be the reduced Gr\"obner basis of $J_G$ with respect to a monomial order $<$ and set $\mathcal{G}_2 = \{ x_S x_T \mid S, T \in S(G),  \ S \cap T \neq \emptyset\}$.
    Then \[ \mathcal{G}=
    (\mathcal{G}_1 \setminus M_G)\cup \mathcal{G}_2
    \]
    is the reduced Gr\"obner basis of $K_G$ with respect to $<$.
\end{Proposition}

\begin{proof}
    Since $\mathcal{G}_1$ is a Gr\"obner basis of $J_G$, we have $J_G=\langle \mathcal{G}_1 \rangle$.
    It follows from $M_G= \langle \mathcal{G}_2 \rangle$ that $\Gc$ is a set of generators of $K_G = J_G + M_G$.
    We apply Buchberger's Criterion to $\Gc$.
    If $p, q \in \mathcal{G}_1 \setminus M_G$, then the remainder of $S(p,q)$ on division by $\Gc$ is $0$ since $\mathcal{G}_1$ is a Gr\"obner basis of $J_G$.
    If $p, q \in \mathcal{G}_2$, then both $p$ and $q$ are monomials,
    and hence $S(p,q)=0$.
    Suppose that $p \in \mathcal{G}_2$ and $q \in \mathcal{G}_1 \setminus M_G$.
    Let $p = x_S x_T$ with $S,T \in S(G)$ and $S \cap T \neq \emptyset$, and let $q = \xb_f - \xb_g$ where $f$ and $g$ are $k$-colorings of $G_{\ab}$ with $\ab \in \ZZ^d_{\geq 0}$ and ${\rm in}_< (q) = \xb_f$.

\medskip

\noindent
    {\bf Case 1.} ($p$ and $\xb_f$ are relatively prime.)
    Then $S(p,q) = p \xb_g$ is divided by $p \in  \mathcal{G}_2$.
    
\medskip

\noindent
    {\bf Case 2.} ($\xb_f$ is divided by $x_S$.)
    From Lemma \ref{belongs to M}, $\xb_f$ is not divided by $x_T$ since $q \notin M_G$.
    Then $S(p,q) = x_T \xb_g$. Let $i \in S \cap T$.
    Since $\xb_f$ is divided by $x_S$, $\pi(\xb_f)=\pi(\xb_g)$ is divided by $t_i$.
    Thus $\pi(x_T \xb_g) $ is divided by $t_i^2$.
    Hence $x_T \xb_g$ is divided by a monomial $x_{S'} x_T$ such that $i \in S' \cap T$.
    Then the remainder of $S(p,q) =  x_T \xb_g$ on division by $\Gc_2$ is $0$.

\medskip

    Therefore the remainder of any $S$-polynomial on division by $\Gc$ is $0$.
    By Buchberger's Criterion, $\Gc$ is a Gr\"obner basis of $K_G$.
    Since $\Gc_1$ is reduced, it is easy to see that $\Gc$ is reduced.
\end{proof}

Similarly to $\langle [I_G]_2 \rangle$ and $J_G$, we can determine if two $k$-colorings of $G$ are Kempe equivalent by $K_G$.

\begin{Proposition} \label{KG and Kempe}
    Let $G$ be a graph on $[d]$ and let $f$ and $g$ be $k$-colorings of an induced subgraph of $G$.
%
Then $f \sim_k g$ if and only if $\xb_f-\xb_g \in K_G$.
\end{Proposition}

\begin{proof}
Let $f$ and $g$ be $k$-colorings of an induced subgraph $G_0$ of $G$.
From Theorem \ref{ExamKempe} it is enough to show that $\xb_f-\xb_g \in J_G$ if and only if $\xb_f-\xb_g \in K_G$.
Since $J_G \subset K_G$, $\xb_f-\xb_g \in K_G$ if $\xb_f-\xb_g \in J_G$.

Suppose that $\xb_f-\xb_g \in K_G$. We show that $\xb_f-\xb_g \in J_G$.
Let $\mathcal{G}$ be the reduced Gr\"obner basis of $K_G$ with respect to a monomial order $<$.
    From Proposition~\ref{kg GB},
    $\mathcal{G} = (\mathcal{G}_1 \setminus M_G)\cup \mathcal{G}_2$ where $\mathcal{G}_1$ is the reduced Gr\"obner basis of $J_G$ with respect to $<$ and $\mathcal{G}_2 = \{ x_S x_T \mid S, T \in S(G),  \ S \cap T \neq \emptyset\}$.
Let $\xb_{f'}$ (resp.~$\xb_{g'}$)
denote the remainder of $\xb_f$ (resp.~$\xb_g$)
on division by $\mathcal{G}_1 \setminus M_G$.
Suppose that $\xb_{f'} \ne \xb_{g'}$.
Since $f$ and $g$ are $k$-colorings of $G_0$, we have $\xb_f, \xb_g \notin M_G$
from Lemma \ref{not in M}.
From Lemma~\ref{belongs to M}, $\xb_{f'}, \xb_{g'} \notin M_G$
since both $\xb_f-\xb_{f'}$ and $\xb_g-\xb_{g'}$ belong to $J_G$.
Thus $\xb_{f'} - \xb_{g'} $ ($\ne 0$)
is the normal form of $\xb_f- \xb_g \in K_G$ with respect to $\mathcal{G}$.
This contradicts that 
 $\mathcal{G}$ is 
 a Gr\"obner basis of $K_G$.
Thus $\xb_{f'} = \xb_{g'}$.
Then 
the normal form of $\xb_f-\xb_g$ with respect to $\mathcal{G}_1 \setminus M_G$ is zero,
and hence $\xb_f-\xb_g \in J_G$.
\end{proof}

Now, we give a proof of Theorem \ref{thm:K-class}.
In fact, Theorem \ref{thm:K-class} follows from the following.

\begin{Theorem} \label{crs}
  Let $G$ be a graph on $[d]$ and $<$ a monomial order on $R[G]$, and let $\{\xb_{f_1},\dots,\xb_{f_s}\}$ be the set of all standard monomials of degree $k$ with respect to ${\rm in}_<(K_G)$.
    Then each $f_i$ is a $k$-coloring of an induced subgraph of $G$.
    In addition, given an induced subgraph $G'$ of $G$,
\[
\{f_1, \dots, f_s\} \cap \mathcal{C}_k(G')
\]
is a complete representative system
    for ${\rm kc}(G',k)$.
\end{Theorem}

\begin{Remark}
    If $\xb_f=x_{S_1} x_{S_2} \cdots x_{S_k}$ is a standard monomial of degree $k$ with respect to ${\rm in}_< (K_G)$, then $f$ is a $k$-coloring of the induced subgraph of $G$ on $\bigcup_{i=1}^{k} S_i$.
\end{Remark}

\begin{proof}[Proof of Theorem \ref{crs}]
Let $\mathcal{G}$ be the reduced Gr\"obner basis of $K_G$ with respect to a monomial order $<$.
    From Proposition~\ref{kg GB},
    $\mathcal{G} = (\mathcal{G}_1 \setminus M_G)\cup \mathcal{G}_2$ where $\mathcal{G}_1$ is the reduced Gr\"obner basis of $J_G$ with respect to $<$ and $\mathcal{G}_2 = \{ x_S x_T \mid S, T \in S(G),  \ S \cap T \neq \emptyset\}$.
Since each $\xb_{f_i}$ is not divided by any monomial in $\mathcal{G}_2$, 
from Lemma~\ref{not in M},
each $f_i$ is a $k$-coloring of an induced subgraph of $G$.


Suppose that $f_i \sim_k f_j$ for some $1 \le i < j \le s$.
From Theorem \ref{ExamKempe}, $\xb_{f_i} - \xb_{f_j}$ belongs to $J_G (\subset K_G)$.
Then ${\rm in}_< (\xb_{f_i} - \xb_{f_j})$ is not standard, a contradiction.
Hence $f_i \not\sim_k f_j$ for any $1 \le i < j \le s$.

Let $f$ be a $k$-coloring of $G'$, and let $\xb_{f'}$ be the remainder of $\xb_f$ on division by $\mathcal{G}_1 \setminus M_G$.
Then $\xb_f - \xb_{f'}$ belongs to $J_G$.
Since $\xb_f \notin M_G$, we have $\xb_{f'} \notin M_G$
from Lemma~\ref{belongs to M}.
Thus $\xb_{f'}$ is the normal form of $\xb_f$ with respect to $\mathcal{G}$, and hence
$\xb_{f'} = \xb_{f_j}$ for some $j$.
Since $\xb_f - \xb_{f_j}$ belongs to $J_G$, 
we have $f \sim_k f_j$ from Theorem \ref{ExamKempe}. 
\end{proof}
We see an example of Theorem \ref{crs}.
\begin{Example}
    {\rm 
    Let $G$ be the graph as in Example \ref{ex:nonkempe}.
We consider the reverse lexicographic order $<$ on $R[G]$ such that
\[
x_{\emptyset} < x_{\{1\}} < \cdots < x_{\{6\}} < x_{\{1,5\}} < x_{\{1,6\}} < x_{\{2,4\}} < x_{\{2,6\}} < x_{\{3,4\}} < x_{\{3,5\}}.
\]
Then there are $65$ standard monomials of degree $3$ with respect to ${\rm in}_<(K_G)$. In particular, the standard monomials $x_{\{1,5\}} x_{\{2,6\}} x_{\{3,4\}}$ and $x_{\{1,6\}} x_{\{2,4\}} x_{\{3,5\}}$ correspond to $3$-colorings of $G$.
It then follows from Theorem \ref{crs} that the associated $3$-colorings, which are the $3$-colorings as in Example~\ref{ex:nonkempe}, form a complete representative system for ${\rm kc}(G,k)$.
}
\end{Example}

As a consequence of Theorem \ref{crs}, the number of $k$-Kempe classes ${\rm Kc}(G,k)$ can be computed by using Hilbert functions.
We denote ${\rm Ind}(G)$ the set of all induced subgraphs of $G$ and denote ${\rm Ind}_m (G)$  the set of all induced subgraphs of $G$ with $m$ vertices.
\begin{Corollary}
\label{cor:K_class}
Let $G$ be a graph on $[d]$.
Then one has
\[
H(R[G]/K_G, k)=\sum_{G' \in {\rm Ind}(G)} {\rm Kc}(G',k).
\]
In particular, 
\[
{\rm Kc}(G,k)= \sum_{m=0}^d (-1)^{d-m} \sum_{G' \in {\rm Ind}_m (G)} H(R[G']/K_{G'},k).
\]
\end{Corollary}
\begin{Example}
    {
    Let $G$ be the graph as in Example \ref{ex:nonkempe}.
    Then one has 
    \[
    H(R[G]/K_G, k)=
    \left\{
    \begin{array}{cc}
1 & k=0,\\
13 & k=1,\\
49 & k=2,\\
65 & k=3,\\
64  & k\ge 4,
    \end{array}
    \right.
    \]
    and 
    \[
    H(R[G']/K_{G'},3)=32
    \]
    for any $G'\in {\rm Ind}_5(G)$.
    Note that ${\rm Kc}(G',k) \ge 1$ for any $k \ge \chi(G)$ and $G' \in {\rm Ind}(G)$. Moreover, one has $|{\rm Ind}(G)| = 2^6 = 64$. Hence we have the following
    from $H(R[G]/K_{G},k)$.
    \begin{itemize}
    \item 
    Since $H(R[G]/K_{G},2) < 2^6$, one has $\chi(G) \ge 3$.
    In fact, $\chi(G) = 3$ in this case.
    \item
    Let $k \ge 4$.
    Since $H(R[G]/K_{G},k) = 2^6$,
    ${\rm Kc}(G',k)=1$ for any $G' \in {\rm Ind}(G)$.
    \item
    Since $H(R[G']/K_{G'},3)=2^5$, ${\rm Kc}(G'',3)=1$ for any $G'' \in {\rm Ind}(G)$ with $G \ne G''$.
    Hence one has ${\rm Kc}(G,3)=2$ from $H(R[G]/K_{G},3) = 65$.
    \end{itemize}
    }
\end{Example}
Let $I$ be a graded ideal of $R$. 
In general, the Hilbert function $H(R/I,k)$ is not always a polynomial. 
However, there exists a unique polynomial $P_{R/I} \in \QQ[k]$ such that $H(R/I,k)=P_{R/I}(k)$ for $k$ large enough. We call $P_{R/I}$ the \textit{Hilbert polynomial} of $R/I$.
We show that $P_{R[G]/K_G}$ is a constant which depends only on the number of vertices.
\begin{Proposition}
    Let $G$ be a graph on $[d]$.
    Then one has
    \[
    P_{R[G]/K_G}(k)=2^d.
    \]
\end{Proposition}
\begin{proof}
    Let $\Delta$ be the maximum degree of $G$. Then it follows from \cite[Corollary 2.5]{Moh} that for any integer $k \geq \Delta +1$ and for any induced subgraph $G'$ of $G$, one has ${\rm Kc}(G',k)=1$.
    Hence for any integer $k \geq \Delta +1$, we obtain
\[
H(R[G]/K_G, k)=\sum_{G' \in {\rm Ind}(G)} {\rm Kc}(G',k)
=\sum_{G' \in {\rm Ind}(G)} 1
=|{\rm Ind}(G)|=2^d.
\]
This implies that 
  \[
    P_{R[G]/K_G}(k)=2^d,
    \]
    as desired.
\end{proof}

\section{Algorithms}
\label{sect:algo}
In this section, by using Theorems \ref{thm:K_equiv} and \ref{thm:K-class} and techniques on Gr\"{o}bner bases, we introduce several algorithms related to Kempe equivalence.

First we give Algorithm \ref{compute GB for KG} to compute the reduced Gr\"obner basis of $K_G$.
From Proposition~\ref{I to I_G}, $\Gc_2$ in Algorithm \ref{compute GB for KG} is a Gr\"obner basis of $I_G$.
It then follows that $\Gc_3$ in Algorithm \ref{compute GB for KG} is a system of generators of $J_G$.
\begin{figure}[h]
\begin{algorithm}[H]
    \caption{Computation of the reduced Gr\"obner basis of $K_G$}
    \label{compute GB for KG}
    \begin{algorithmic}[1]   
    \REQUIRE The set of all stable sets of $G$, and a monomial order $<$ on $R[G]$.
    \ENSURE The reduced Gr\"obner basis of $K_G$ with respect to $<$.
    \STATE
    $\Fc := \{x_{S \setminus \{i\}} x_{\{i\}}- x_{S} x_\emptyset \mid i \in S \in S(G), \ |S| \ge 2\}$.
    \STATE
    From $\Fc$, compute the reduced Gr\"obner basis $\Gc_1$ of $L_G$ with respect to a reverse lexicographic order such that $x_S  \geq x_{\emptyset}$ for any $S \in S(G)$.
    \STATE
    $\mathcal{G}_2:=\{ g /x_{\emptyset}^k \mid g \in \mathcal{G}_1, k \in \ZZ_{\geq 0}$, $x_{\emptyset}^k$ divides $g$, $x_{\emptyset}^{k+1}$ does not divide $g \}$.
    \STATE
    $\mathcal{G}_3:=\{ x_{S_1} x_{S_2} - x_{S_3} x_{S_4}  \in \Gc_2 \mid  S_1, S_2, S_3, S_4 \in S(G), \ S_1 \cap S_2 =S_3 \cap S_4 =\emptyset \}$.
    \STATE
    Compute the reduced Gr\"obner basis $\Gc$ of $K_G$ with respect to $<$ from a system of generators
    $\Gc_3 \cup \{x_S x_T \mid S, T \in S(G), S \cap T \ne \emptyset\}$.
    \RETURN
    $\Gc$.
    \end{algorithmic}
\end{algorithm}
\end{figure}
Using the reduced Gr\"obner basis of $K_G$,
Algorithm \ref{one} determines whether $f \sim_k g$ or not.
The correctness of Algorithm \ref{one} is an immediate consequence of Lemma~\ref{IMP}
and Proposition~\ref{KG and Kempe}.
From Theorem~\ref{ExamKempe}, it is possible to replace $K_G$ with either $\langle [I_G]_2\rangle$ or $J_G$
in Algorithm \ref{one}.
\begin{figure}[!t]
\begin{algorithm}[H]
    \caption{Determination of the Kempe equivalence}
    \label{one}
    \begin{algorithmic}[1]   
    \REQUIRE $k$-colorings $f$ and $g$ of $G$, and the reduced Gr\"obner basis $\mathcal{G}$ of $K_G$.
    \ENSURE ``$f \sim_k g$" or ``$f \not\sim_k g$".
    \STATE
Compute the normal form $m$ of $\xb_f - \xb_g$
with respect to $\mathcal{G}$.
\IF {$m=0$} 
\RETURN ``$f  \sim_k g$"
\ELSE 
\RETURN ``$f \not\sim_k g$"
\ENDIF
    \end{algorithmic}
\end{algorithm}
\end{figure}
On the other hand, by the reduced Gr\"obner basis of $K_G$, we can compute the set of all standard monomials of degree $k$ with respect to the initial ideal ${\rm in}_<(K_G)$.
From Theorem~\ref{thm:K-class}, we have Algorithm~\ref{kanzen} to compute a complete representative system for ${\rm kc}(G,k)$.
\begin{figure}[!t]
\begin{algorithm}[H]
    \caption{Computation of a complete representative system for ${\rm kc}(G,k)$}
    \label{kanzen}
    \begin{algorithmic}[1]   
    \REQUIRE The reduced Gr\"obner basis $\mathcal{G}$ of $K_G$.
    \ENSURE A complete representative system for ${\rm kc}(G,k)$.
    \STATE
    From $\Gc$, compute the set $\{\xb_{f_1},\dots,\xb_{f_s}\}$ of all standard monomials of degree $k$ with respect to the initial ideal ${\rm in}_<(K_G)$.
    \STATE
    $C:=\{\}.$
\FOR{$i=1,2,\dots,s$} 
\IF{$\xb_{f_i}=x_{S_1} \cdots x_{S_k}$ satisfies $[d]=S_1 \cup \cdots \cup S_k$}
\STATE
$C:=C \cup \{f_i\}$.
\ENDIF
\ENDFOR
\RETURN $C$
    \end{algorithmic}
\end{algorithm}
\end{figure}
Algorithm~\ref{enumerate} enumerates all elements in a Kempe equivalent class.
The correctness of Algorithm \ref{enumerate} is guaranteed by 
Theorem \ref{crs} together with the fact that,
with respect to a Gr\"obner basis,
the normal form of the monomials in the same residue class is unique.
In addition, Algorithm \ref{enumerate} is based on \cite[Algorithm 5.7]{Stu}
for enumeration of fibers using a Gr\"obner basis of a toric ideal.
As stated in \cite{Stu}, ``reverse search" technique 
is useful to improve the efficiency.
\begin{figure}[!t]
\begin{algorithm}[H]
    \caption{Enumeration of elements in a Kempe equivalent class}
    \label{enumerate}
    \begin{algorithmic}[1]   
    \REQUIRE $k$-coloring $f$ of $G$, and the reduced Gr\"obner basis $\mathcal{G}$ of $K_G$.
    \ENSURE All $k$-colorings of $G$ which are Kempe equivalent to $f$ (up to permutations of colors).
    \STATE Compute the normal form $\xb_{f'}$ of $\xb_f$ with respect to $\mathcal{G}$.
    \STATE $A:=\{f'\}$, $B:=\{\}$.
    \WHILE{$A\ne \{\}$}
\STATE
    Choose $g \in A$.
    \FOR{$\xb_{g_1} - \xb_{g_2}\in {\mathcal G}$ with $ \xb_{g_1} > \xb_{g_2}$}
\IF{$\xb_{g_2}$ divides $\xb_g$ \AND $g' \notin B$ where $\xb_{g'} = \xb_g \xb_{g_1}/ \xb_{g_2}$} 
\STATE
$A:=A \cup \{g'\}$.
\ENDIF
\ENDFOR
\STATE
$A:= A \setminus \{g\}$, $B:=B \cup \{g\}$.
\ENDWHILE
    \RETURN $B$.
    \end{algorithmic}
\end{algorithm}
\end{figure}
Finally, we give an algorithm to find a sequence of Kempe switchings.
We define a {\it Kempe basis} which is the set of sequences of colorings corresponding to the reduced Gr\"obner basis of $K_G$.

\begin{Definition}
Work with the same notation as in Theorem \ref{kg GB},
that is, $\mathcal{G}=
    (\mathcal{G}_1 \setminus M_G)\cup \mathcal{G}_2
    $
    is the reduced Gr\"obner basis of $K_G$ with respect to $<$.
Let $\mathcal{G}_1 \setminus M_G = \{\xb_{p_1} - \xb_{q_1} , \dots,\xb_{p_t} - \xb_{q_t}\}$.
From Theorem \ref{thm:K_equiv}, $p_j$ and $q_j$ are Kempe equivalent for each $j$.
Then 
\[
\left\{
(f_1^{(1)},\dots, f_{s_1}^{(1)}), 
\dots, 
(f_1^{(t)},\dots, f_{s_t}^{(t)})
\right\}
\]
is called a {\it Kempe basis} of $G$ with respect to $<$
if $f_1^{(j)} = p_j$, $f_{s_j}^{(j)} = q_j$, and $f_{i}^{(j)}$ is obtained from $f_{i-1}^{(j)}$ by a Kempe switching.
\end{Definition}

In order to give an algorithm to find a Kempe basis, the following three Procedures are important.

\bigskip

\noindent
\textbf{Procedure 1.}
($\xb_{f} - \xb_{g} = x_{S_1} x_{S_2} - x_{S_3} x_{S_4} \in J_G
\mapsto $
a sequence of Kempe switching from $f$ to $g$.)

Suppose that 
$f$ and $g$ are 2-colorings of an induced subgraph $G_0$ of $G$.
Let $\xb_{f} - \xb_{g} = x_{S_1} x_{S_2} - x_{S_3} x_{S_4} \in J_G$.
Then $S_1 \cap S_2 = S_3 \cap S_4 = \emptyset$ and $S_1 \cup S_2 = S_3 \cup S_4$.
Hence $g$ is obtained from $f$ by setting
$$
g(x) =
\begin{cases}
    f(x) & x \notin H,\\
    1 & x \in H \mbox{ and } f(x) =2,\\
    2 & x \in H \mbox{ and } f(x) =1,
\end{cases}
$$
where $H$ is the induced subgraph of $G_0$ on the vertex set
$(S_1 \setminus S_3) \sqcup (S_2 \setminus S_4)$.
Suppose that $H$ has $p$ connected components $H_1,\dots, H_p$.
For $l=1,2,\dots,p$, let $f_0 = f$ and let
$$
f_l(x) =
\begin{cases}
    f(x) & x \notin H,\\
    1 & x \in H_1 \cup \dots \cup H_l \mbox{ and } f(x) =2,\\
    2 & x \in H_1 \cup \dots \cup H_l \mbox{ and } f(x) =1.
\end{cases}
$$
Then $f_0,f_1,\ldots,f_p$ is a sequence of $2$-colorings of $G_0$
such that $f_0 = f$, $f_p = g$, and
$f_i$ is obtained from $f_{i-1}$ by a Kempe switching.

\bigskip

\noindent
\textbf{Procedure 2.}
Let $f'$ and $g'$ be $k$-colorings of an induced subgraph $G_1$ of $G$.
Suppose that $\xb_{f'} - \xb_{g'}= {\bf x}_{h} (\xb_{f}- \xb_{g})$.
Then $f$ and $g$ are the restrictions of $f'$ and $g'$ to some induced subgraph $G_2$ of $G_1$, respectively. 
Hence if  $f_0,f_1,\ldots,f_p$ is a sequence of $k'$-colorings of  $G_2$ such that $f_0 = f$, $f_p = g$, and $f_i$ is obtained from $f_{i-1}$ by a Kempe switching, then
 $f_0',f_1',\ldots,f_p'$ satisfy $f_0' = f'$, $f_p' = g'$, and
$f_i'$ is obtained from $f_{i-1}'$ by a Kempe switching, 
where  $f_i'$ is a $k$-coloring of $G_1$ obtained by combining $f_i$ with $h$.

\bigskip

\noindent
\textbf{Procedure 3.}
Let $\xb_{p} - \xb_{q}$ be a binomial in the reduced Gr\"{o}bner basis $\Gc$ of $K_G$ with respect to a monomial order $<$.
From Proposition \ref{kg GB}, any binomial $\xb_p - \xb_q \in \Gc$ is obtained from the reduced Gr\"obner basis of $J_G$
with respect to $<$
which is computed from
\[
\{\xb_f - \xb_g \mid \mbox{$f$ and $g$ are $2$-colorings of an induced subgraph of $G$}\}
\]
by Buchberger's Algorithm.
By keeping track of the computation of Buchberger's Algorithm, we can compute an expression
$$
\xb_{p} - \xb_{q}  = \sum_{r=1}^s {\bf x}_{w_r} (\xb_{f_r}- \xb_{g_r}),
$$
where $f_r$ and $g_r$ are $2$-colorings of an induced subgraph of $G$ for each $r$,
\[
{\bf x}_{w_r} \xb_{g_r} = {\bf x}_{w_{r+1}} \xb_{f_{r+1}}
\]
for each $r=1,\dots,s-1$, and
$\xb_{p} = {\bf x}_{w_1} \xb_{f_1}$, $\xb_{q} = {\bf x}_{w_s} \xb_{g_s}$.
See \cite[Chapter 2, \S 1]{AL} for details.
Then we can compute a sequence of Kempe switchings from 
$u_r$ to $v_r$ where $\xb_{u_r} = \xb_{w_r} \xb_{f_r}$ and $\xb_{v_r} = \xb_{w_r} \xb_{g_r}$ by Procedures 1 and 2.
Combining them, we have a sequence of Kempe switchings from $p$ to $q$.

\bigskip

Using Procedures 1, 2, and 3, we have Algorithm \ref{Kempe basis} that computes a Kempe basis.
Note that, if the reduced Gr\"obner basis of $K_G$ consists of quadratic polynomials,
then we can skip large part of the procedure in Algorithm \ref{Kempe basis}.

\begin{figure}[!t]
\begin{algorithm}[H] 
    \caption{Construction of a Kempe basis}
    \label{Kempe basis}
    \begin{algorithmic}[1]   
    \REQUIRE A monomial order $<$ on $R[G]$.
    \ENSURE A Kempe basis of $G$ with respect to $<$.
\STATE
Compute a system of (quadratic) generators 
\[
\Fc=\{\xb_{f_1}- \xb_{g_1}, \dots, \xb_{f_t}- \xb_{g_t}\} \cup \{x_S x_T \mid S, T \in S(G) , S\cap T \ne \emptyset \}
\]
of $K_G$.
\FOR{$r=1,2,\dots,t$}
\STATE
By Procedure 1, compute a sequence $F_r$ of 2-colorings corresponding to
a sequence of Kempe switchings from $f_r$ to $g_r$. 
\ENDFOR
\STATE
Compute the reduced Gr\"obner basis $\Gc$ of $K_G$ with respect to $<$ from $\Fc$.
\STATE
$A:=\{\}$.
\FOR{$\xb_{p} - \xb_{q} \in \mathcal{G}$}
\STATE
By keeping track of the computation of $\mathcal{G}$ from $\Fc$,
compute the expression
$$
\xb_{p} - \xb_{q}  = \sum_{r=1}^s {\bf x}_{w_r} (\xb_{f_{i_r}}- \xb_{g_{i_r}}),
$$
where 
 ${\bf x}_{w_r} \xb_{g_{i_r}} = {\bf x}_{w_{r+1}} \xb_{f_{i_{r+1}}}$
for each $r=1,\dots,s-1$,
and
$\xb_{p} = {\bf x}_{w_1} \xb_{f_{i_1}}$, $\xb_{q} = {\bf x}_{w_s} \xb_{g_{i_s}}$.
\FOR{$r=1,2,\dots,s$}
\STATE
$f := f_{i_r}'$ where $\xb_{f_{i_{r}}'} = {\bf x}_{w_r} \xb_{f_{i_r}}$,
and $g:= g_{i_r}'$ where $\xb_{g_{i_r}'} = {\bf x}_{w_r} \xb_{g_{i_r}}$.
\STATE
By extending $F_{i_r}$,
find a sequence of colorings $h_1^{(r)}, \dots, h_{t_r}^{(r)}$ where $h_1^{(r)} = f$, $h_{t_r}^{(r)}= g$ and
$h_j^{(r)}$ is obtained from $h_{j-1}^{(r)}$ by a Kempe switching.
\ENDFOR
\STATE
$A:= A \cup \{(h_1^{(1)}, \dots, h_{t_1}^{(1)} (= h_1^{(2)}), h_2^{(2)}, \dots, h_{t_2}^{(2)}
( = h_1^{(3)}), \dots, h_{t_s}^{(s)} )\}$
\ENDFOR
\RETURN
$A$
    \end{algorithmic}
\end{algorithm}
\end{figure}

We now use a Kempe basis to find a sequence of Kempe switchings between two colorings.
Suppose that the normal form $\xb_{f'}$ of $\xb_f$ with respect to the reduced Gr\"obner basis $\mathcal{G}$ of $K_G$
is given by 
$$\xb_f =\xb_{f_0}
\underset{h_1}{\rightarrow}
\xb_{f_1}
\underset{h_2}{\rightarrow}
\cdots
\underset{h_t}{\rightarrow}
\xb_{f_t}=
\xb_{f'}, \ \ \ (h_j \in \mathcal{G}).
$$
By Procedure 2, we can construct a sequence of Kempe switchings from $f_{i-1}$ to $f_i$
by extending a sequence of Kempe switchings corresponding to $h_{i-1}$ in a Kempe basis.
Using this fact, we have Algorithm \ref{alg1} to find a sequence of Kempe switchings.

Although Algorithm \ref{compute GB for KG} requires the enumeration of all stable sets and the computation of a Gr\"obner basis, and thus is not computationally feasible for large graphs, its merit lies in providing a uniform algebraic framework. In particular, the method is mechanically applicable and conceptually clarifies the structure of the graph from an abstract algebraic viewpoint.

\begin{figure}[!t]
\begin{algorithm}[H]
    \caption{Construction of a sequence of Kempe switchings}
    \label{alg1}
    \begin{algorithmic}[1]   
    \REQUIRE $k$-colorings $f$ and $g$ of $G$, the reduced Gr\"obner basis $\mathcal{G}$ of $K_G$,
    and a Kempe basis $\mathcal{K}$ of $G$.
    \ENSURE ``a sequence of Kempe switchings from $f$ to $g$"
    or ``$f \not\sim_k g$".
    \STATE
Compute the normal form $\xb_{f'}$ of $\xb_f$
with respect to $\mathcal{G}$.
    \STATE
Compute the normal form $\xb_{g'}$ of $\xb_g$
with respect to $\mathcal{G}$.
\IF {$\xb_{f'} =\xb_{g'}$} 
\STATE By Procedure 2, from $\mathcal{K}$, find a sequence of colorings $f=f_1,\dots, f_s,f'$,
where $f_i$ is obtained from $f_{i-1}$ by a Kempe switching.
\STATE By Procedure 2, from $\mathcal{K}$, find a sequence of colorings $g=g_1,\dots, g_t , g'$ ($=f'$),
where $g_i$ is obtained from $g_{i-1}$ by a Kempe switching.
\RETURN $f_1,\dots, f_s,f',g_t ,g_{t-1},\dots,g_1$. 
\ELSE 
\RETURN ``$f \not\sim_k g$"
\ENDIF
    \end{algorithmic}
\end{algorithm}
\end{figure}

\subsection*{Acknowledgment}
This work was supported by JSPS KAKENHI 24K00534 and 22K13890.


\begin{thebibliography}{9}



\bibitem{AL}
W. Adams and P. Loustaunau,
``An Introduction to Gr{\"o}bner Bases",
Graduate Studies in Mathematics {\bf 3},
American Mathematical Society, 1994.


\bibitem{CLO}
D. A. Cox, J. Little and D. O'Shea, ``Ideals, Varieties, and Algorithms", fourth edition, Undergraduate Texts in Mathematics,
Springer, 2015.


\bibitem{HHO}
J. Herzog, T. Hibi and H. Ohsugi, \textit{Binomial ideals}, Graduate Text in Mathematics, Springer, 2018.

\bibitem{MOS}
K. Matsuda, H. Ohsugi and K. Shibata,
Toric rings and ideals of stable set polytopes,
{\it Mathematics} {\bf 7} (2019),  613.


\bibitem{Moh}
B. Mohar,
Kempe Equivalence of Colorings,
{\em in} ``Graph Theory in Paris" (A.~Bondy, J.~Fonlupt, JL.~Fouquet, JC.~Fournier, J.~L.~Ram\'irez Alfons\'in, Eds), Trends in Mathematics, 
pp.~287--297,
Birkh\"auser Basel, 2006.

\bibitem{OST}
H. Ohsugi, K. Shibata and A. Tsuchiya,
Perfectly contractile graphs and quadratic toric rings,
{\em Bull. Lond. Math. Soc.} {\bf 55} (2023), 1264--1274.

\bibitem{OTkempe}
H. Ohsugi and A. Tsuchiya,
Kempe equivalence and quadratic toric rings, preprint, \\
{\tt arXiv:2303.12824}


\bibitem{Stu}
B. Sturmfels, 
``Gr\"{o}bner bases and convex polytopes", 
Amer. Math. Soc., Providence, RI, 1996. 

\end{thebibliography}
\end{document}